\documentclass{amsart}

\usepackage{url,amssymb,enumerate,colonequals}

\usepackage{mathrsfs} 
\usepackage{MnSymbol}
\usepackage{extarrows}
\usepackage[all,cmtip]{xy}

\usepackage[OT2,T1]{fontenc}
\usepackage{color}

\usepackage[
        colorlinks, citecolor=darkgreen,
        backref,
        pdfauthor={Nuno Freitas}, 
]{hyperref}

\usepackage{comment}
\usepackage{multirow}

\newcommand{\F}{\mathbb{F}}

\newcommand{\Q}{\mathbb{Q}}

\newcommand{\Z}{\mathbb{Z}}

\newcommand{\fq}{\mathfrak{q}}

\newcommand{\calO}{\mathcal{O}}

\newcommand{\fp}{\mathfrak{p}}

\DeclareMathOperator{\Gal}{Gal}

\DeclareMathOperator{\Norm}{Norm}

\DeclareMathOperator{\ord}{ord}

\numberwithin{equation}{section}

\newtheorem{theorem}{Theorem}
\newtheorem*{thm}{Theorem}
\newtheorem{lemma}{Lemma}[section]

\theoremstyle{definition}
\newtheorem{remark}[lemma]{Remark}

\newtheorem{example}[equation]{Example}

\definecolor{darkgreen}{rgb}{0,0.5,0}

\DeclareRobustCommand{\SkipTocEntry}[5]{}

\begin{document}

\title[Asymptotic Fermat]{On asymptotic Fermat over
$\Z_p$-extensions of~$\Q$}

\author{Nuno Freitas}
\address{Departament de Matem\`atiques i Inform\`atica,
Universitat de Barcelona (UB),
Gran Via de les Corts Catalanes 585,
08007 Barcelona, Spain}
\email{nunobfreitas@gmail.com}

\author{Alain Kraus}
\address{Sorbonne Universit\'e,
Institut de Math\'ematiques de Jussieu - Paris Rive Gauche,
UMR 7586 CNRS - Paris Diderot,
4 Place Jussieu, 75005 Paris, 
France}
\email{alain.kraus@imj-prg.fr}

\author{Samir Siksek}

\address{Mathematics Institute\\
	University of Warwick\\
	CV4 7AL \\
	United Kingdom}

\email{s.siksek@warwick.ac.uk}

\date{\today}
\thanks{Freitas is supported by a Ram\'on y Cajal fellowship (RYC-2017-22262).
Siksek is supported by  
EPSRC grant \emph{Moduli of Elliptic curves and Classical Diophantine Problems}
(EP/S031537/1).}

\begin{abstract}
 Let $p$ be a prime and let $\Q_{n,p}$ denote the $n$-th layer of the cyclotomic
$\Z_p$-extension of $\Q$.  
We prove
the effective asymptotic FLT over $\Q_{n,p}$ for all $n \ge 1$ and all
primes $p \ge 5$ that are non-Wieferich, i.e. 
$2^{p-1} \not \equiv 1 \pmod{p^2}$. 
The effectivity in our result builds 
on recent work of Thorne proving modularity of elliptic curves over $\Q_{n,p}$.
\end{abstract}


\maketitle

\section{Introduction}
Let $F$ be a totally real number field. The \textbf{asymptotic Fermat's Last Theorem over $F$} is the statement that there exists a 
constant~$B_F$, depending only on~$F$, such that, for all primes~$\ell > B_F$, the only solutions to the equation
$x^\ell + y^\ell + z^\ell = 0$, with $x$, $y$, $z \in F$ are the trivial ones satisfying $xyz =0$.
If $B_F$ is effectively computable, we refer to this as the \textbf{effective Fermat's Last Theorem over $F$}. 
Let $p$ be a prime, $n$ a positive integer and write $\Q_{n,p}$
for the $n$-th layer of the cyclotomic $\Z_p$-extension. 
In \cite{FKS}, the authors established the following theorem.
\begin{theorem}\label{thm:Z2}
The effective asymptotic Fermat's Last Theorem
 holds over each layer $\Q_{n,2}$ 
of the cyclotomic $\Z_2$-extension.
\end{theorem}
The proof of Theorem~\ref{thm:Z2} relies heavily on class field theory
and the theory of $2$-extensions, and the method 
depends crucially on the fact that $2$ is totally ramified in $\Q_{n,2}$.
In this paper we establish the following.
\begin{theorem}\label{thm:Zp}
Let $p \ge 5$ be a prime. Suppose $p$ is non-Wieferich, 
i.e. $2^{p-1} \not \equiv 1 \pmod{p^2}$.
The effective asymptotic Fermat's Last Theorem
holds over each layer $\Q_{n,p}$ of the cyclotomic $\Z_p$-extension of $\Q$.
\end{theorem}
We remark that the only  Wieferich primes currently
known are $1093$ and $3511$. It is fascinating to observe
that these primes originally arose in connection with
historical attempts at proving Fermat's Last Theorem.
Indeed Wieferich  \cite{Wieferich} showed
that if $2^{p-1} \not \equiv 1 \pmod{p^2}$
then the first case of Fermat's Last Theorem holds for exponent $p$.

In contrast to Theorem~\ref{thm:Z2}, the proof of Theorem~\ref{thm:Zp}
makes use of a criterion (Theorem~\ref{thm:FS} below) 
established in \cite{FS1}
for asymptotic FLT in terms of solutions
to a certain $S$-unit equation. 
The proof of that criterion builds on many
deep results
including modularity lifting theorems
due to Breuil, Diamond, Gee, Kisin, and others, and 
Merel's uniform boundedness theorem, and exploits  
the strategy of Frey, Serre, Ribet, Wiles and 
Taylor, utilized in Wiles' proof~\cite{Wiles}
of Fermat's Last Theorem. In this paper
we use elementary arguments
to study these $S$-unit equations in $\Q_{n,p}$ and this
study, together with the $S$-unit criterion,
 quickly yields Theorem~\ref{thm:Zp}.
The effectivity in Theorem~\ref{thm:Zp} builds on the following
great theorem due to Thorne \cite{Thorne}.
\begin{thm}[Thorne]
Elliptic curves over $\Q_{n,p}$ are modular.
\end{thm}

\section{An $S$-unit criterion for asymptotic FLT}
The following criterion for asymptotic FLT
is a special case of \cite[Theorem 3]{FS1}.
\begin{theorem}\label{thm:FS}
Let $F$ be a totally real number field. Suppose the Eichler--Shimura conjecture over $F$ holds.
Assume that $2$ is inert in $F$ and write $\fq=2\calO_F$ 
for the prime ideal above $2$. Let $S=\{\fq\}$
and write $\calO_S^\times$ for the group of $S$-units in $F$. Suppose every
solution $(\lambda,\mu)$ to the $S$-unit equation 
\begin{equation}\label{eqn:Sunit}
\lambda+\mu=1, \qquad \lambda,~\mu \in \calO_S^\times
\end{equation}
satisfies both of the following conditions
\begin{equation}\label{eqn:conds}
\max\{ \lvert \ord_\fq(\lambda) \rvert,~ \lvert \ord_\fq(\mu) \rvert\} \le 4, \qquad
\ord_{\fq}(\lambda \mu) \equiv 1 \pmod{3}.
\end{equation}
Then the asymptotic Fermat's Last Theorem holds over $F$.
Moreover, if all elliptic curves over $F$ with 
full $2$-torsion are modular, then
the effective asymptotic Fermat's Last Theorem holds over $F$.
\end{theorem}
For a discussion of the Eichler--Shimura conjecture see  
\cite[Section 2.4]{FS1}, but for the purpose of this paper
we note that the conjecture is known to hold for all totally real fields of odd
degree. In particular, it holds for $\Q_{n,p}$ for all odd $p$.


To apply Theorem~\ref{thm:FS} to $F=\Q_{n,p}$ we need to know
for which $p$ is $2$ inert in $F$. The answer is given by the
following lemma, which for $n=1$ is Exercise 2.4 in \cite{Washington}.
\begin{lemma} \label{lem:inertZp}
Let $p \ge 3$, $q$ be distinct primes. 
Then $q$ is inert in $\Q_{n,p}$
if and only if $q^{p-1}\not \equiv 1 \pmod
{p^2}$.
\end{lemma} 
\begin{proof}
Let $L=\Q(\zeta_{p^{n+1}})$ 
 and $F=\Q_{n,p}$. Write $\sigma_q$ and $\tau_q$
for the Frobenius elements corresponding to $q$
in $\Gal(L/\Q)$ and $\Gal(F/\Q)$. 
The prime $q$ is inert in $F$ precisely when $\tau_q$ has order $p^n$. 
The natural surjection $\Gal(L/\Q) \rightarrow \Gal(F/\Q)$
sends $\sigma_q$ to $\tau_q$ and its kernel has order $p-1$.
Thus $q$ is inert in $F$ if and only if the
order of $\sigma_q$ is divisible by $p^n$,
which is equivalent to  $\sigma_q^{p-1}$ having order
$p^n$.
There is
a canonical isomorphism
$\Gal(L/\Q) \rightarrow (\Z/p^{n+1} \Z)^\times$
sending $\sigma_q$ to $q+p^{n+1} \Z$. 
Thus $q$ is inert in $F$ if and only if $q^{p-1}+p^{n+1}\Z$ 
has order $p^n$. This is equivalent to $q^{p-1} \not \equiv 1 \pmod{p^2}$.
\end{proof}

\section{Proof of Theorem~\ref{thm:Zp}}
\begin{lemma}\label{lem:norm}
Let $\fp$ be the unique prime  
above $p$ in $F=\Q_{n,p}$. Let $\lambda \in \calO_F$. Then $\lambda \equiv \Norm_{F/\Q}(\lambda)
\pmod{\fp}$.
\end{lemma}
\begin{proof}
As $p$
is totally ramified in $F$, we know that the residue
field $\calO_F/\fp$ is $\F_p$. Thus there is some
$a \in \Z$ such that $\lambda \equiv a \pmod{\fp}$.
Let $\sigma \in G=\Gal(F/\Q)$. Since $\fp^\sigma=\fp$,
we have $\lambda^\sigma \equiv a \pmod{\fp}$. Hence
\[
\Norm_{F/\Q}(\lambda)=\prod_{\sigma \in G} \lambda^\sigma
\equiv a^{\#G} \pmod{\fp}.
\] 
However $\#G=p^n$ so 
$\Norm_{F/\Q}(\lambda) \equiv a \equiv \lambda \pmod{\fp}$. 
\end{proof}

\begin{lemma}\label{lem:Zpunit}
Let $p \ne 3$ be a rational prime. 
Let $F=\Q_{n,p}$.
Then the unit equation 
\begin{equation}\label{eqn:unit}
 \lambda + \mu = 1, \quad \lambda, \; \mu \in \calO_F^\times
\end{equation}
has no solutions.
\end{lemma}
\begin{proof}
Let $(\lambda,\mu)$ be a solution to \eqref{eqn:unit}.
By Lemma~\ref{lem:norm}, $\lambda \equiv \pm 1 \pmod{\fp}$
and $\mu \equiv \pm 1 \pmod{\fp}$. Thus
$\pm 1 \pm 1 \equiv \lambda+\mu =1 \pmod{\fp}$.
This is impossible as $p \ne 3$.
\end{proof}

\begin{remark}\label{remark}
Lemma~\ref{lem:Zpunit}
is false for $p=3$. Indeed, Let $p=3$ and $n=1$.
Then $F=\Q_{1,3}=\Q(\theta)$ where $\theta$ satisfies
$\theta^3 - 6\theta^2 + 9\theta - 3=0$. 
The unit equation has solution $\lambda=2-\theta$
and $\mu=-1+\theta$. In fact, the 
unit equation solver of the computer algebra system
\texttt{Magma}~\cite{MAGMA} gives a total of $18$ solutions.
\end{remark}

\begin{lemma}\label{lem:ZpSunit}
Let $p \ge 5$ be a rational prime. Let $F=\Q_{n,p}$.
Suppose $2$ is inert in $F$ and write $\fq=2\calO_F$
for the unique prime above $2$. Let $S=\{\fq\}$
and write $\calO_S^\times$ for the group of $S$-units.
Then every solution to the $S$-unit equation \eqref{eqn:Sunit}
satisfies one of the following:
\begin{enumerate}
\item[(i)] $\ord_\fq(\lambda)=1$, $\ord_\fq(\mu)=0$;
\item[(ii)] $\ord_\fq(\lambda)=0$, $\ord_\fq(\mu)=1$;
\item[(iii)] $\ord_\fq(\lambda)=\ord_\fq(\mu)=-1$.
\end{enumerate}
\end{lemma}
\begin{proof}
Write $n_\lambda=\ord_\fq(\lambda)$ and $n_\mu=\ord_\fq(\mu)$.
Suppose first $n_\lambda \ge 2$.
Then $n_\mu=0$ and so $\mu \in \calO_F^\times$.
Moreover, as $4 \mid \lambda$, we have $\mu \equiv 1 \pmod{4}$
and so $\mu^\sigma \equiv 1 \pmod{4}$ for all $\sigma \in G=\Gal(F/\Q)$.
Hence $\Norm_{F/\Q}(\mu) = \prod \mu^\sigma \equiv 1 \pmod{4}$. But
$\Norm_{F/\Q}(\mu)=\pm 1$, thus $\Norm_{F/\Q}(\mu)=1$. As before, denote
the unique prime above $p$ by $\fp$. By Lemma~\ref{lem:norm}
we have $\mu \equiv 1 \pmod{\fp}$. Hence $\fp$ divides
$1-\mu=\lambda$ giving a contradiction. 

Thus $n_\lambda \le 1$.
Next suppose
 $n_\lambda \le -2$.  Then $n_\lambda=n_\mu$. 
Let
$\lambda^\prime=1/\lambda$ and $\mu^\prime=-\mu/\lambda$.
Then $(\lambda^\prime,\mu^\prime)$ is a solution to the 
$S$-unit equation satisfying $n_{\lambda^\prime} \ge 2$,
giving a contradiction by the previous case. 
Hence $-1 \le n_\lambda \le 1$ and by symmetry 
$-1 \le n_\mu \le 1$. From
Lemma~\ref{lem:Zpunit} either $n_\lambda \ne 0$ or $n_\mu \ne 0$.
Thus one of (i), (ii), (iii) must hold.
\end{proof}

\begin{remark}
Possibilities (i), (ii), (iii)
cannot be eliminated because of the solutions $(2,-1)$, $(-1,2)$
and $(1/2,1/2)$ to the $S$-unit equation.
\end{remark}

\begin{proof}[Proof of Theorem~\ref{thm:Zp}]
We suppose $p\ge 5$ and non-Wieferich. It follows from Lemma~\ref{lem:inertZp}
that $2$ is inert in $F=\Q_{n,p}$. Write $\fq=2\calO_F$.
By Lemma~\ref{lem:ZpSunit}
all solutions $(\lambda,\mu)$ to the $S$-unit equation \eqref{eqn:Sunit}
satisfy \eqref{eqn:conds}.
We now apply Theorem~\ref{thm:FS}.
As elliptic curves over $\Q_{n,p}$ are modular thanks to 
Thorne's theorem, we conclude that the effective Fermat's
Last Theorem holds over $\Q_{n,p}$. 
\end{proof}

\begin{remark}
The proof of Theorem~\ref{thm:Zp}
for $p=3$ and for the Wieferich primes seems out of reach
at present. There are solutions to the unit equation 
in $\Q_{1,3}$ (as indicated in Remark~\ref{remark}), 
and therefore in $\Q_{n,3}$ for all $n$, and
these solutions violate the criterion of Theorem~\ref{thm:FS}.
For $p$ a Wieferich prime, $2$ splits in $\Q_{n,p}$
into at least $p$ prime ideals and we would 
need to consider the $S$-unit equation \eqref{eqn:Sunit}
with $S$ the set of primes above $2$. 
It appears difficult to treat the $S$-unit equation
in infinite families of number fields where $\#S \ge 2$
(c.f.\ \cite[Theorem 7]{FKS} and its proof).
\end{remark}

\section{A Generalization}
In fact, the proof of Theorem~\ref{thm:Zp}
establishes the following more general theorem.
\begin{theorem}
Let $F$ be a totally real number field 
and $p \ge 5$ be a rational prime. Suppose that the following conditions
are satisfied.
\begin{enumerate}
\item[(a)] $F$ is a $p$-extension of $\Q$
(i.e. $F/\Q$ is a Galois extension of degree $p^n$ for some $n \ge 1$).
\item[(b)] $p$ is totally ramified in $F$.
\item[(c)] $2$ is inert in $F$.
\end{enumerate}
Then the asymptotic Fermat's Last Theorem holds for $F$.
\end{theorem}
\begin{example}
A quick search on the L-Functions and Modular Forms Database \cite{LMFDB}
yields $153$ fields of degree $5$ satisfying
conditions of the theorem with $p=5$. 
The one with smallest discriminant is $\Q_{1,5}$. The one
with the next smallest discriminant is $F=\Q(\theta)$
where 
$\theta^5 - 110\theta^3 - 605\theta^2 - 990\theta - 451=0$.
The discriminant of $F$ is $5^8 \cdot 11^4$.
It is therefore not contained in any $\Z_p$-extension of $\Q$.
\end{example}


\begin{thebibliography}{}

\bibitem{MAGMA} W.\ Bosma, J.\ Cannon et C.\ Playoust: \emph{The Magma
Algebra System I: The User Language}, J.\ Symb.\ Comp.\ \textbf{24} (1997),
235--265. (see also \texttt{http://magma.maths.usyd.edu.au/magma/})  

\bibitem{FS1} N. Freitas and S. Siksek, 
\emph{The asymptotic Fermat's Last Theorem for five-sixths of real quadratic
fields},
 Compositio Mathematica  \textbf{151} (2015), 1395--1415.
 
\bibitem{FKS} N. Freitas, A. Kraus and  S. Siksek, 
\emph{Class field theory, Diophantine analysis and the  asymptotic Fermat's Last Theorem},
Advances in Mathematics \textbf{363} (2020), to appear.

\bibitem{LMFDB} The LMFDB Collaboration, 
\emph{The L-functions and Modular Forms Database}, 
http://www.lmfdb.org, 2020, [Online; accessed 24 January 2020].

\bibitem{Thorne}
J. Thorne,
\emph{Elliptic curves over $\Q_\infty$ are modular},
Journal of the European Mathematical Society \textbf{21} (2019), 1943--1948.

\bibitem{Washington} L.\ C.\ Washington,
\emph{Introduction to Cyclotomic Fields},
second edition, GTM \textbf{83}, Springer, 1997.

\bibitem{Wieferich}
A.\ Wieferich, 
\emph{Zum letzten Fermatschen Theorem},
Journal f\"{u}r die reine und angewandte Mathematik \textbf{136} (1909), 
293--302. 

\bibitem{Wiles} A.\ Wiles,
{\em Modular elliptic curves and Fermat's Last Theorem},
Annals of Math.\ {\bf 141} (1995), 443--551.

\end{thebibliography}
\end{document}